\documentclass[reqno, 11pt, final]{amsart}
\usepackage{bbding,textcomp,amsfonts,amsthm,amsmath,mathrsfs, verbatim}
\usepackage{multirow, url}
\usepackage[utf8x]{inputenc}
\usepackage{a4wide}

\newtheorem{theorem}{Theorem}

\newtheorem{lemma}[theorem]{Lemma}
\newtheorem{claim}{Claim}
\newtheorem{remark}[theorem]{Remark}
\newtheorem{observation}[theorem]{Observation}

\newtheorem{problem}{Problem}
\newcommand{\eps}{\varepsilon}

\newcommand{\prob}{\textrm{Pr}}
\newcommand{\Crln}{C^{(r,\l)}_n}
\newcommand{\rex}{\textrm{\upshape{rex}}}
\newcommand{\ex}{\textrm{\upshape{ex}}}

\newcommand{\GF}{\mathrm{GF}}

\newcommand{\rank}{\textrm{\upshape{rk}}\,}

\title{On the rank of higher inclusion matrices}

\PrerenderUnicode{\unichar{355}}

\author{Codru\unichar{355} Grosu, Yury Person, Tibor Szab\'o}
\thanks{The first author was supported by the Berlin Mathematical School (BMS), a graduate school in the 'Initiative of Excellence', 
and by the Deutsche Forschungsgemeinschaft within the research training group `Methods for Discrete Structures' (GRK 1408).}

\address{Institut f\"{u}r Mathematik, Freie Universit\"at Berlin, Arnimallee 3-5, D-14195 Berlin, Germany}
\email{grosu.codrut@gmail.com, person\,\textbar{}\,szabo@math.fu-berlin.de}

\begin{document}

\begin{abstract}
Let $r \geq s \geq 0$ be integers and $G$ be an $r$-graph. The higher inclusion matrix $M_s^r(G)$ is a $\{0,1\}$-matrix with rows indexed 
by the edges of $G$ and columns indexed by the subsets of $V(G)$ of size $s$: the entry corresponding to an edge $e$ and a subset $S$ is $1$ 
if $S \subseteq e$ and $0$ otherwise. Following a question of Frankl and Tokushige and a result of Keevash, we define the rank-extremal function 
$\rex(n, t, r, s)$ as the maximum number of edges of an $r$-graph $G$ having $\rank M_s^r(G) \leq \binom{n}{s} - t$. For $t$ at most linear in $n$ 
we determine this function as well as the extremal $r$-graphs. The special case $t=1$ answers a question of Keevash.
\end{abstract}
\maketitle

\section{\large{Introduction}}

Let $n \geq 1$ and suppose $\mathcal{F}$ is a collection of $k$-subsets of $[n]$. For any $p \geq 1$, 
we can define the lower $p$-shadow $\partial^p_l \mathcal{F}$ of $\mathcal{F}$ as the set of all $(k-p)$-subsets of 
$[n]$ which are contained in at least one element of $\mathcal{F}$. If $p=1$, we may drop the superscript. A 
fundamental result in extremal combinatorics, the Kruskal-Katona theorem, gives a sharp lower bound for the size of $\partial^p_l \mathcal{F}$. In order 
to state the theorem, we note that for positive integers $m$ and $k$ there are always unique integers $m_k > m_{k-1} > \ldots > m_j \geq j > 0$ such 
that $m = \binom{m_k}{k}+\binom{m_{k-1}}{k-1}+\ldots+\binom{m_j}{j}$. We further make the convention that $\binom{x}{y} = 0$ whenever $y < 0$.

\begin{theorem}[\cite{Kat68}, \cite{Kru63}]
\label{thm:katona}
Let $k \geq 1, p \geq 1$ and $m \geq 1$. For every $\mathcal{F} \subseteq \binom{[n]}{k}$ with $m = |\mathcal{F}|$ we have
\begin{equation}
\label{eq:kruskal}
|\partial^p_l \mathcal{F}| \geq \binom{m_k}{k-p}+\binom{m_{k-1}}{k-1-p}+\ldots+\binom{m_j}{j-p}.
\end{equation}
The inequality is best possible for every $k, p$ and $m \leq \binom{n}{k}$. Furthermore, 
if $m = \binom{m_k}{k}$, then equality holds in \eqref{eq:kruskal} if and only if $\mathcal{F} \simeq \binom{[m_k]}{k}$.
\end{theorem} 

Theorem~\ref{thm:katona} also provides a sharp lower bound for the size of the upper $p$-shadow $\partial^p_u \mathcal{F}$, 
which is the set of all $(k+p)$-subsets of $[n]$ which contain at least one element of $\mathcal{F}$ 
(indeed, $\partial^p_u \mathcal{F} = (\partial^p_l \mathcal{F}^c)^c$, where $\mathcal{A}^{c}$ is the family obtained by taking the complements 
of all the sets in $\mathcal{A}$). 
For later use, we shall denote
\begin{equation*}
K(n, m, k, p) = \min\,\left\{|\partial_u^p \mathcal{F}| \: : \: \mathcal{F}
\subseteq \binom{[n]}{k}, \ |\mathcal{F}|=m \right\}.
\end{equation*}
Equivalently, as $|\partial_u^p \mathcal{F}| = |\partial^p_l \mathcal{F}^c|$, $K(n, m, k, p)$ denotes the minimum 
size of $|\partial_l^p \mathcal{F}|$, taken over all collections $\mathcal{F} \subseteq \binom{[n]}{n-k}$ of size $m$. 
By the Kruskal-Katona theorem, if $m = \binom{m_{n-k}}{n-k}+\binom{m_{n-k-1}}{n-k-1}+\ldots+\binom{m_j}{j}$ is the unique decomposition of $m$, then 
\begin{equation*}
K(n, m, k, p) = \binom{m_{n-k}}{n-k-p}+\binom{m_{n-k-1}}{n-k-1-p}+\ldots+\binom{m_j}{j-p}.
\end{equation*}

A weaker but simpler version of the Kruskal-Katona theorem is due to Lov\'asz (\cite{Lovasz79}, Exercise $13.31(b)$): 
if $|\mathcal{F}|=\binom{x}{k}$, for some real number $x$, then $|\partial^p_l \mathcal{F}|\geq \binom{x}{k-p}$, with 
equality if and only if $x$ is an integer and $\mathcal{F} \simeq \binom{[x]}{k}$. Several algebraic versions of Lov\'asz's result 
exist \cite{Chow2010}, \cite{Keevash08}; however, no analogue of the full Kruskal-Katona theorem in linear algebra has so far been obtained.

In this paper we will be mainly concerned with the algebraic generalization due to Keevash \cite{Keevash08}. To state it we will first introduce several definitions.

A hypergraph $G$ is an ordered pair of sets $(V, E)$, where $E \subseteq 2^V$, the elements of $V$ are called vertices, and the elements of $E$ are called edges. 
If $r \geq 1$ and all elements of $E$ have the same size $r$, we say $G$ is an $r$-uniform hypergraph, or simply an $r$-graph. Hence $2$-graphs correspond 
to the usual notion of undirected graphs. We let $v(G)$ denote the number of vertices, or the order of $G$, and $|G|$ the number of edges, or the size of $G$. 
We denote by $K_n^r$  the complete $r$-graph on $n$ vertices and vertex set $[n]$. For a subset $E'\subseteq E$  of the edges of an $r$-graph $G=(V,E)$ we shall 
denote by $G-E$ the hypergraph with vertex set $V$ and edge set $E\setminus E'$. 
Furthermore if $F \subseteq \binom{[n]}{r}$ we shall sometimes identify $F$ with the $r$-uniform hypergraph $([n], F)$. 
Finally, if $G = (V, E)$ is any $r$-graph we let $\overline{G}$ denote
its complement, i.e. the hypergraph 
$\left(V, \binom{V}{r} \setminus E\right)$. 

Now let $G$ be an $r$-graph and $s \leq r$. The higher inclusion matrix $M_s^r(G)$ is a $\{0,1\}$-matrix with rows indexed by the edges of $G$ and columns 
indexed by the subsets of $V(G)$ of size $s$: the entry 
corresponding to an edge $e$ and a subset $S$ is $1$ if $S \subseteq e$ and $0$ otherwise. Thus $M_1^2(G)$ is the usual incidence matrix of a graph $G$. 

It is an open problem of Frankl and Tokushige \cite{FranklToku91} to determine the minimum rank of $M_s^r(G)$ in terms of $|G|$. 
A theorem of Gottlieb \cite{Gottlieb} shows that the matrix $M_s^r(K_n^r)$ has full rank. 

\begin{theorem}[Gottlieb, \cite{Gottlieb}]
\label{thm:gottlieb}
For every $n\geq r\geq s\geq 0$ we have 
$$\rank M_s^r(K_n^r) = \min\left\{\binom{n}{r}, \binom{n}{s}\right\}.$$
\end{theorem}

Here and in the following the rank is considered only over the reals (in fact, we can take any field of characteristic $0$). 
Now the generalization of Lov\'asz's theorem due to Keevash is the following.

\begin{theorem}[\cite{Keevash08}]
\label{thm:keevash}
For every $r \geq s \geq 0$ there is a number $n_{r,s}$ so that if $G$ is an $r$-graph with $|G| = \binom{x}{r} \geq n_{r,s}$ then 
$\rank{M_s^r(G)} \geq \binom{x}{s}.$ Furthermore, if $r > s > 0$ then equality holds if and only if $x$ is an integer and $G\simeq K^r_x$.
\end{theorem}

Theorem~\ref{thm:keevash} implies Lov\'asz's result, as the rank of $M_s^r(G)$ is at most the number of non-zero columns, which is the size 
of the lower $(r-s)$-shadow of $E(G)$. An important step in the proof of Theorem~\ref{thm:keevash} was provided by the following lemma.

\begin{lemma}[\cite{Keevash08}]
\label{lem:keevash}
Suppose $n \geq 2r+s$. If $F \subseteq \binom{[n]}{r}$ with $|F| < \binom{r}{s}^{-1}\binom{n}{r-s}$ then $\rank M_s^r(K_n^r - F) = \binom{n}{s}$.
\end{lemma}

Keevash further asked (\cite{Addendum}, Question 2 (iii)) whether Lemma \ref{lem:keevash} remains true under the assumption $|F| < \binom{n-s}{r-s}$, 
at least for large $n$. This would be best possible, as removing all edges of $K_n^r$ containing some fixed $s$-set creates a $0$-column and hence reduces 
the rank by at least $1$.

Now let $r > s \geq 1, n \geq r+s$ and $0 \leq t \leq \binom{n}{s}$. In view of the above and the question of Frankl and Tokushige \cite{FranklToku91}, 
we define the \textit{rank-extremal function} 
\begin{equation*}
\rex(n, t, r, s) := \max\left\{|G| : G \textrm{ is an $r$-graph on $[n]$ and }\rank M_s^r(G) \leq \binom{n}{s} - t\right\}.
\end{equation*}
This notion is analogous to the notion of Tur\'an number of hypergraphs (see \cite{Keevash11} for a survey), where the maximum 
number of edges is sought when the clique number is bounded. Here, instead, the rank of the higher-inclusion matrix is bounded.
 
In the current paper we investigate this function for small $t$ and further determine the \textit{rank-extremal} $r$-graphs. As in the case of $t=1$, 
there is a natural construction which provides a lower bound for given $t$: fixing $t$ subsets of $[n]$ of size $s$ and removing all edges of $K_n^r$ 
containing at least one of them yields an $r$-graph $G$ with $\rank M_s^r(G) \leq \binom{n}{s} - t$, as $M_s^r(G)$ has at least $t$ zero columns. 
Minimizing over all choices of the $t$ $s$-subsets, we see that 
\begin{equation*}
\rex(n, t, r, s) \geq \binom{n}{r}-K(n, t, s, r-s),
\end{equation*}
where recall that $K(n, t, s, r-s)$ is the smallest size
of the upper $(r-s)$-shadow of a collection of $s$-subsets of $[n]$ of size $t$, provided by the Kruskal-Katona theorem. 

The content of our main theorem is that equality holds here for $t$ at most linear in $n$.

\section{\large{New Results}} 
\label{sec:statements}

We call a collection of $s$-sets an \textit{$s$-star configuration} if their common intersection has size at 
least $s-1$. Clearly for any $t \leq n - s + 1$, there exists a 
unique (up to isomorphism) $s$-star configuration $\mathcal{S}_{n,t,s}$ of size $t$ on ground set $[n]$. 
It is easy to see that $|\partial_u^{r-s} \mathcal{S}_{n,t,s}| = \binom{n-s+1}{r-s+1} - \binom{n-s+1-t}{r-s+1}$. 
Removing all edges of $K_n^r$ containing some element of $\mathcal{S}_{n,t,s}$ gives an $r$-graph $G(n, t, r, s)$ with $\rank M_s^r(G(n, t, r, s)) \leq \binom{n}{s} - t$, 
as there are at least $t$ zero columns. Then we have
\begin{equation}
\label{eq:oper}
\rex(n, t, r, s) \geq \binom{n}{r}-K(n, t, s, r-s) \geq e(G(n, t, r, s)) = \binom{n}{r} - \binom{n-s+1}{r-s+1} + \binom{n-s+1-t}{r-s+1}. 
\end{equation}

Our main result is the following.

\begin{theorem}
\label{thm:main}
Let $r > s \geq 1$. Then there exist positive constants $c_0 := c_0(r, s)$ and $n_0 := n_0(r, s)$ such that the following holds. 
Let $n \geq n_0$ and $1 \leq t \leq c_0n$ be integers. Then
\begin{equation}
\label{eq:main}
\rex(n, t, r, s) = \binom{n}{r} - \binom{n-s+1}{r-s+1} + \binom{n-s+1-t}{r-s+1}.
\end{equation}
Furthermore, $G(n, t, r, s)$ is the unique rank-extremal $r$-graph up to isomorphism.
\end{theorem}

In particular, equality holds everywhere in \eqref{eq:oper}, and so $K(n, t, s, r-s) = \binom{n-s+1}{r-s+1} - \binom{n-s+1-t}{r-s+1}$. 
Now note that for any $t \leq n-s$ we have
\begin{equation}
\label{eq:tDecomp}
t = \binom{n-s}{n-s}+\ldots+\binom{n-s-t+1}{n-s-t+1},
\end{equation}
and so by unicity, the right-hand side of \eqref{eq:kruskal} (for $k:=n-s$, $p:=r-s$ and $m:=t$) 
is exactly $\binom{n-s}{n-r}+\ldots+\binom{n-s-t+1}{n-r-t+1} = \binom{n-s+1}{r-s+1} - \binom{n-s-t+1}{r-s+1}$. 
Thus we have obtained a linear-algebraic proof of Theorem~\ref{thm:katona} for the lower $(r-s)$-shadow of a collection of $(n-s)$-sets of size at most $c_0n$. 

By taking $t:=1$ in Theorem \ref{thm:main}, we obtain $\rex(n,1,r,s)= \tbinom{n}{r} - \tbinom{n-s}{r-s}$ which gives a positive answer to Keevash's question.

The main step in the proof of Theorem~\ref{thm:main}, which might be interesting in its own right, is an extension of Gottlieb's Theorem and
concerns the robustness or resilience of the rank. It states that if a limited number of somewhat uniformly distributed $r$-sets are removed from the 
complete $r$-graph then the rank of the $s$-inclusion matrix does not decrease. 

\begin{theorem}
\label{thm:robust}
For all $r \geq s \geq 0$ there exist positive constants $\eps :=\eps(r, s), \alpha := \alpha(r, s)$ and $n_1 := n_1(r, s)$ such that the following holds. 
If $F\subseteq \tbinom{[n]}{r}$ is a family of at most $\eps n \tbinom{n-s}{r-s}$ $r$-sets on $n \geq n_1$ vertices, with the property that every $s$-set 
$S \in \tbinom{[n]}{s}$ is contained in less than $\alpha \tbinom{n-s}{r-s}$ elements of $F$, then 
$\rank M_s^r(K_n^r - F) = \tbinom{n}{s}$.
\end{theorem}

Theorem~\ref{thm:main} implies that for $t$ up to some small constant times $n$, we have $\binom{n}{r}-\rex(n, t, r, s) = K(n, t, s, r-s)$. One may wonder what is the maximum value $t_{\textrm{max}}$ up to which this equality holds, provided $n$ is large enough. Let us define 
\begin{equation*}
t_{\textrm{max}} (n,r,s) := \max\left\{1 \leq t \leq \binom{n}{s} : \textrm{ for any $t' \leq t$,} \binom{n}{r}-\rex(n, t', r, s) = K(n, t', s, r-s)\right\}.
\end{equation*} 
Our next theorem shows that
our bound on $t$ in Theorem~\ref{thm:main} is best possible
up to a constant factor.
\begin{theorem}
\label{thm:sharp}
For arbitrary  integers $r > s \geq 1$ 
there exists a positive constant $n_0' := n_0'(r, s)$ such that 
for every $n \geq n_0'$, $t_{\textrm{max}}(n,r,s) < n - r - 1$.
\end{theorem}

The rest of this paper is organized as follows. In Section 3 we prove Theorem~\ref{thm:robust}. 
In Section 4 we prove the main theorem. Finally, in Section 5 we prove Theorem \ref{thm:sharp}.

\textbf{Remark.} After completion of this work we were informed by Harout Aydinian that Ahlswede, Aydinian and Khachatrian considered before the problem of determining $\rex(n, t, r, 1)$ and solved 
it in \cite{Ahls03} for all values of $n, t$ and $r$. Their proof technique is very different from ours. 
In \cite{Ahls03} it is also mentioned that the function $\rex(n, 1, r, 1)$ was first studied by Longstaff \cite{Longstaff77}, with the complete determination of its value being made by Odlyzko \cite{Odlyzko81}.

\section{\large{Resilience of the rank}} 
\label{sec:robust}

The main goal of this section is to prove Theorem~\ref{thm:robust}.

For an arbitrary $r$-graph $H$ and vertex $x \in V(H)$ we define two derived hypergraphs.
We shall denote by $H/x$ the $(r-1)$-graph on vertex set $V(H)-\{x\}$ with 
edge set $\{A \setminus \{x\} : A \in E(H), x\in A\}$ and denote by $H-x$ 
the $r$-graph  on vertex set $V(H)-\{x\}$ with edge set $\{A :  A \in
E(H), x\notin A\}$.
Observe that for every $x\in V(H)$, $|H| = |H/x| +|H-x|$.

Let $G$ be any $r$-graph with vertex set $V$. For any $s$-subset $S \subseteq V$ and any $r$-edge $e \in E(G)$, we shall denote by $c_S(G)$ the column corresponding to $S$ in $M_s^r(G)$, and by $c_S(G)_e$ the entry of $M_s^r(G)$ corresponding to the edge $e$ and the $s$-subset $S$. 
A sequence $\{\alpha_S\}_{S \in \binom{V}{s}}$ of real numbers is called 
a \textit{dependence sequence} for $G$ if
\begin{equation}
\label{eq:dependence}
\sum_{S \in \binom{V}{s}} \alpha_S c_S(G)_e = 0,
\end{equation}
for every $e \in E(G)$.
If not all the coefficients $\alpha_S$ are $0$, 
then the columns of $M_s^r(G)$ are linearly dependent and we call the sequence \textit{non-trivial}. 
For a dependence sequence we construct the \textit{associated $s$-graph} $G'$ on vertex set $V(G')=V(G)$ 
with edge set  $E(G') = \{S\in \tbinom{V}{s}  : \alpha_S\neq 0\}$.

For later use we collect here four simple, but useful observations about an arbitrary dependence sequence 
$\{\alpha_S\}_{S\in \tbinom{V}{s}}$ and its associated $s$-graph $G'$.
\begin{observation}
\label{obs:f}
If $V' \subseteq V(G)$ and $H := G[V']$ is the induced $r$-graph on $V'$ then $\{\alpha_S\}_{S \in \binom{V'}{s}}$ is a dependence sequence for $H$.
\end{observation}
\begin{proof}
The proof follows by noting that the submatrix  of $M_s^r(G)$ indexed by the rows from $E(H)=\binom{V'}{r}\cap E(G)$ and the columns from $\binom{V}{s} \setminus \binom{V'}{s}$ has only zero entries, and hence when \eqref{eq:dependence} is applied to $G$ and $e \in E(H)$, all terms  of the form $\alpha_S c_S(G)_e$ with $S \notin \binom{V'}{s}$ vanish. That is, \eqref{eq:dependence} holds for $H$ and $e$ as well.
\end{proof}
An $r$-subset $R\subseteq V(G')$ is called a {\em $1$-clique} if the induced subgraph $G'[R]$ contains exactly one edge of $G'$.
\begin{observation}
\label{obs:s}
No edge of $G$ is a $1$-clique.
\end{observation}
\begin{proof}
Indeed, if $R\in E(G)$ is an $r$-set containing exactly one edge $S'$ of $G'$, then $\alpha_S = 0$ for any $S \in \binom{R}{s} - \{S'\}$.  
Applying \eqref{eq:dependence} to $G$ and $R$ we obtain $\alpha_{S'} = \alpha_{S'}c_{S'}(G)_R= 0$, so $S'$ is not an edge of $G'$, a 
contradiction.
\end{proof}
\begin{observation}
\label{obs:t}
If $x \in V(G)$ and $E(G' - x)$ is empty, then the sequence $\beta_{S'} := \alpha_{\{x\} \cup S'}, S' \in \binom{V(G/x)}{s-1}$ is a dependence sequence for $G/x$.
\end{observation}
\begin{proof}
Note that the condition of $E(G' - x)$ being empty implies $\alpha_S = 0$, for any $s$-set $S$ not containing $x$. Then for any $e \in E(G/x)$ we have
\begin{align*} 
\sum_{S' \in \binom{V(G/x)}{s-1}} \beta_{S'} c_{S'}(G/x)_e 
&= \sum_{S' \in \binom{V(G/x)}{s-1}} \alpha_{\{x\} \cup S'} c_{\{x\} \cup S'}(G)_{\{x\} \cup e}, 
\quad \textrm{by definition,}\\
&= \sum_{S \in \binom{V(G)}{s}, S\ni x} \alpha_{S} c_S(G)_{\{x\} \cup e},\\
&= \sum_{S \in \binom{V(G)}{s}} \alpha_S c_S(G)_{\{x\} \cup e}, 
\quad \textrm{as $\alpha_S = 0$ when $x \notin S$,}\\
&= 0, \quad \textrm{by \eqref{eq:dependence} applied to $G$ and $\{x\} \cup e$.} 
\end{align*}
So \eqref{eq:dependence} holds for $G/x$ and $\{\beta_{S'}\}$ as well.
\end{proof}

Our last observation extends the notion of a $1$-clique. 
A \textit{$k$-semistar with center $x$ and set of leaves $L$} is an induced subgraph of $G'$ on the $(k+1)$-set of vertices $\{x\} \cup L$, 
such that $L$ is an independent set in $G'$ of size $k$ 
and at least one of the $s$-sets $\{x\} \cup S', S' \in \binom{L}{s-1},$ is present in $E(G')$.

\begin{observation}
\label{obs:semistar}
For any $(r+s-2)$-semistar $Z$ in $G'$, with center $x$ and set of leaves $L$, at least one of the $r$-edges $\{x\} \cup R', R' \in \binom{L}{r-1},$ must be in $F$. 
\end{observation}
\begin{proof}
If $s=1$, then $V(Z)$ is a $1$-clique, because $\{ x \} \in E(G')$ and 
$V(Z)\setminus \{ x\}$ is independent in $G'$. Hence $V(Z)$ is in 
$F$ by Observation~\ref{obs:s} and the claim holds.

So we may assume for a contradiction that $s \geq 2$ and every edge $\{x\} \cup R', R' \in \binom{L}{r-1},$ is in $G$. Consider the $(r-1)$-graph $H$ induced by $G/x$ on $L$. Then $H \simeq K^{r-1}_{r+s-2}$ is complete. Hence $M_{s-1}^{r-1}(H)$ has full rank by Gottlieb's Theorem. 

On the other hand, we can first apply Observation \ref{obs:f} to $G$ and $V(Z)$ followed by Observation \ref{obs:t} to $x$ and $G[V(Z)]$, and conclude that the sequence defined by $\beta_{S'} := \alpha_{\{x\} \cup S'}, S' \in \binom{L}{s-1},$ is a dependence sequence for $H = G[V(Z)]/x$. Since $Z$ is a semistar, at least one of the $s$-sets $\{x\} \cup S', S' \in \binom{L}{s-1},$ must be present in $E(G')$, hence $\beta_{S'} \neq 0$ for some $S' \in \binom{L}{s-1}$.
In conclusion, the dependence sequence is non-trivial and this contradicts the columns of $M_{s-1}^{r-1}(H)$ being linearly independent.
\end{proof}

Before we turn to the proof of Theorem~\ref{thm:robust}, we show a
lemma ensuring the existence of a large independent set in $G'$ provided
the rank drops below $\binom{n}{s}$  upon the deletion of a not so large family $F$ of edges.
\begin{lemma}\label{lem:stable_set}
  Let $F\subseteq E(K^r_n)$ be such that $\rank M^r_s(K^r_n- F)<\tbinom{n}{s}$ and 
$(n-r-s)\left(\tbinom{n}{r}-|F|\right)\ge|F|\tbinom{r}{s}\tbinom{n-r+s}{s}$, then any associated $s$-graph $G'$ has 
an independent set of size at least 
\[
 n-\frac{|F|\tbinom{r}{s}\tbinom{n-r+s}{s}}{\tbinom{n}{r}-|F|}.
\]
\end{lemma}
\begin{proof}
Let $G:=K^r_n- F$. To begin with, we find an edge $R\in E(G)$
intersecting not too many $r$-sets $e\in F$ in at least $r-s$ elements.
We estimate the number of such pairs $(R,e)$ (where $R\in E(G)$, $e\in
F$ with $|e \cap R| \geq r-s$) from above by
$|F|\binom{r}{s}\binom{n-r+s}{s}$. Indeed, we can  choose 
$e \in F$ in $|F|$ ways, then specify a subset of $e$ of size $r-s$ in
$\binom{r}{r-s}$ ways, 
and extend it to an element 
$R \in \binom{V(G)}{r} \setminus F$ in at most $\binom{n-r+s}{s}$ ways.
By averaging, we find a set $R\in E(G)$ for which there are at most
\[
\frac{|F|\tbinom{r}{s}\tbinom{n-r+s}{s}}{\tbinom{n}{r}-|F|}
\]
$r$-sets $e\in F$ with $|e\cap R|\ge r-s$. 

We now fix one such edge $R\in E(G)$, an arbitrary non-trivial dependence
sequence $\{ \alpha_S\}_{S\in \binom{V}{s}}$, and  deduce that the associated $s$-graph $G'$ has a large independent set. 
For every edge $e \in F$ with $|e \cap R| \geq r-s$, we choose an arbitrary vertex $v_e \in e \setminus R$. 
This is possible, as $e$ is also an $r$-subset and $e\neq R$. Now let $A := V(G')\setminus \{v_e : e \in F \textrm{ and } |e \cap R| \geq r-s\}$. 
Note that $R \subseteq A$. By the assumption of the lemma,
we also have that 
\[
|A|\ge n-\tfrac{|F|\tbinom{r}{s}\tbinom{n-r+s}{s}}{\tbinom{n}{r}-|F|} \geq r+s. 
\]

We claim that $A$ is independent in $G'$.
Suppose for a contradiction that an $s$-set $Q \subseteq A$ is an edge in $G'$. Form a set $V' \subseteq A$ on $r+s$ vertices containing $R$ and $Q$. This is possible, as $|R\cup Q| \leq r+s$ and there are enough vertices in $A$. By construction, no $r$-set $e\in F$ is contained in $V'$, as otherwise $|e \cap R| \geq |e| - |V'\setminus R| = r-s$ hence $v_e$ is defined and $v_e \in e \subseteq V' \subseteq A$, a contradiction. Therefore $G= K_n^r - F$ induces on $V'$ a complete $r$-graph $\binom{V'}{r}$, which by Gottlieb's theorem must have an inclusion matrix $M_s^r(G[V'])$ of full rank. In particular, every dependence sequence for $G[V']$ must be trivial. 

However, by Observation \ref{obs:f}, the coefficients $\alpha_T, T \in \binom{V'}{s},$ form a dependence sequence for the $r$-graph $G[V']$, 
which is non-trivial, as $Q \subseteq V'$ and $\alpha_Q \neq 0$. This is a contradiction, completing the proof of the lemma. 
\end{proof}

We are now ready to prove Theorem~\ref{thm:robust}.

\begin{proof}[Proof of Theorem~\ref{thm:robust}.]
Note that for $r=s$ or $s=0$, we may simply choose $\eps := 1, n_1 := r$ and $\alpha := 1$. Hence we may assume $r > s > 0$.

We first set some positive constants, depending on $r$ and $s$:
\begin{align}
\label{eps} \eps & < \min\left\{\frac{(r-s)!}{2r!}, \frac{1}{4\binom{r}{s}^2}, \frac{1}{r!2^{2r-s}}\right\},\\
\label{delta} \delta &= 2\eps \binom{r}{s}^2,\\
\label{alpha} \alpha & < \left(\frac{1-2\delta}{2(r-s)}\right)^{r-s}, \\
\label{n_1} n_1 & = \left\lceil \frac{4r}{\frac{1}{2}-\delta}\right\rceil.
\end{align}

\noindent Note that, by~\eqref{eps}, $\delta <1/2$, hence $\alpha$ and $n_1$ can  indeed be chosen to be positive.

Suppose for a contradiction that Theorem~\ref{thm:robust} is false. Then there exists a family $F \subset E(K_n^r)$ of $r$-subsets
on $n \geq n_1$ vertices, with $|F| < \eps n \binom{n-s}{r-s}$ and every $s$-set $S \subseteq V(G)$ is contained in less than $\alpha \binom{n-s}{r-s}$ members of $F$,
such that for $G := K_n^r - F$ we have $\rank M_s^r(G) < \binom{n}{s}$. 
Let $\{\alpha_S\}_{S\in \tbinom{V}{s}}$ be a non-trivial dependence sequence for $G$ and let $G'$ be the associated $s$-graph. Note that $E(G')$ is non-empty, as the sequence is non-trivial. 

Since $s \geq 1$ and $\eps < \frac{(r-s)!}{2r!}$ we have that
\begin{equation*}
\label{eq:binom1}
\binom{n}{r} > 2\eps n\binom{n-s}{r-s} > 2|F|.
\end{equation*}
Then 
$$\frac{|F|\binom{r}{s}\binom{n-r+s}{s}}{\binom{n}{r} -|F|}
\leq 2\eps
n\binom{n-s}{r-s}\binom{r}{r-s}\binom{n-r+s}{s}\binom{n}{r}^{-1}
\leq 2\eps \binom{r}{s}^2n = \delta n,$$
so by Lemma~\ref{lem:stable_set}, $G'$ contains an independent set $A$ of maximum size at least $(1-\delta)n$. 
 Define $B:= V(G') \setminus A$. Then $B$ is non-empty, because $\{\alpha_S\}_{S \in \binom{V(G)}{s}}$ is non-trivial.

In the following we will reach the desired contradiction by showing that the size of $F$ is larger than assumed.

For $x \in B$ we call $S' \in \binom{A}{s-1}$ an \textit{$x$-leaf} if $\{x\} \cup S' \in E(G')$. 
Let $d_A(x)$ be the number of $x$-leaves. We also let $\mu_A(x)$ be the number of vertices in $A$ 
covered by the $x$-leaves. Note that for $s > 1$ we have $d_A(x) \geq \frac{\mu_A(x)}{s-1} \geq 1$, 
since otherwise $A \cup \{x\}$ would be a larger independent set, contradicting the maximality of $A$.

First assume that for some $x \in B$ we have $\mu_A(x) \leq n/2$. Fix this $x$ and let $S_x$ be any $x$-leaf. Note that for $s=1$, we have $S_x = \emptyset$.

If $T$ is any $(r-s)$-subset of $A$ not intersecting any $x$-leaf then $\{x\} \cup S_x \cup T$ forms a $1$-clique, 
because $S_x \cup T \subseteq A$ is an independent set. By Observation~\ref{obs:s}, this shows that there are at 
least $\binom{|A|-\mu_A(x)}{r-s}$ elements of $F$ containing the $s$-set $\{x\} \cup S_x$. 
This is a contradiction to the choice of $F$, 
because 
\begin{equation*}
\binom{|A|-\mu_A(x)}{r-s} \geq \left( \frac{|A|-\mu_A(x)}{r-s}\right)^{r-s} \geq \left(\frac{1-2\delta}{2(r-s)}\right)^{r-s}n^{r-s} > \alpha n^{r-s} > \alpha\binom{n-s}{r-s}.
\end{equation*}
Here we use the fact that $|A|-\mu_A(x) \geq (\frac{1}{2}-\delta)n_1 \geq 4r \geq r-s$ by \eqref{n_1}.

Hence for any $x \in B$ we have that $\mu_A(x) > n/2$. In particular, $s > 1$.

We double-count the triples $(x,S,R)$ where $x\in B$, $S$ is an $x$-leaf, and $R\in \tbinom{A}{r-1}$ with $R\cap S=\emptyset$.
Note that for any such triple, $\{x \} \cup S \cup R$ is an $(r+s-2)$-semistar with center $x$, and their number is
$\sum_{x \in B} d_A(x) \binom{|A|-s+1}{r-1}$. 
On the other hand, by Observation~\ref{obs:semistar}, for any $(r+s-2)$-star $Z$ with center $x$ we can fix an $r$-set $R_Z$ 
such that $x\in R_Z$ and $R_Z\in \tbinom{V(Z)}{r} \cap F$.
Now, for any such $r$-set $R^* \in F$ the number of those triples $(x,S,R)$ for which $R_{\{x\}\cup S \cup R} = R^*$ is
at most $\binom{|A|-r+1}{s-1}$. Note that given $R^*$, the center $x$ is determined uniquely, provided $|R^*\cap B| =1$ (otherwise
there is no appropriate triple at all).

These imply the following lower bound for the size of $F$:
\begin{align}
|F| &\geq \sum_{x \in B} d_A(x) \binom{|A|-s+1}{r-1} \binom{|A|-r+1}{s-1}^{-1} \nonumber \\
&\geq \left(\sum_{x \in B} d_A(x) \right) \frac{1}{(r-1)!}\left(|A|-s-r+3\right)^{r-1} \frac{(s-1)!}{|A|^{s-1}} \nonumber \\
&\geq \frac{(s-1)!}{(r-1)!2^{r-1}} |A|^{r-s} \left(\sum_{x \in B} d_A(x) \right), \textrm{ as $|A| \geq 2(r+s-3)$,} \nonumber \\
&\geq \frac{(s-1)!}{(r-1)!2^{r-1}} |A|^{r-s} \left(\sum_{x \in B} \frac{\mu_A(x)}{s-1} \right) \nonumber 
\\
&\geq \frac{(s-2)!}{(r-1)!2^{2r-s}} n^{r-s+1}, \ \textrm{because
  $\mu_A(x) \geq \frac{n}{2}$ and hence $|A| \geq n/2$.} \nonumber\\
&> \eps n^{r-s+1}, \  \textrm{by~\eqref{eps}} \nonumber\\ 
&\geq \eps n\binom{n-s}{r-s} \nonumber
\end{align}

This contradicts the assumption on the size of $F$
and finishes the proof of Theorem~\ref{thm:robust}. 
\end{proof}

\begin{remark}
The argument above can be adapted to show that in the case $t=1$ of
Theorem~\ref{thm:main} $n_0$ can be taken $2^{5r}$.
\end{remark}

\section{\large{Proof of the main result}}
\label{sec:main}

For any $r \geq s \geq 1$ and $0 \leq t \leq n$ we define 
\begin{equation}
\label{eq:Ndef}
N(n, t, r, s) := \binom{n-s+1}{r-s+1} - \binom{n-s+1-t}{r-s+1}.
\end{equation}
Then for any $s > 1$ it holds that
\begin{equation}
\label{eq:mono}
N(n-1, t, r-1, s-1) = N(n, t, r, s).
\end{equation}
Further note that
\begin{equation}
\label{eq:redef}
N(n, t, r, s) = \sum_{i=1}^{t}\binom{n-s+1-i}{r-s}.
\end{equation}
We shall need the following two technical results about the behaviour of $N(\cdot)$.
\begin{lemma}
\label{lem:Nspeed}
Let $r > s \geq 1$ and $\alpha \in (0, 1)$. Then there exist positive constants $\gamma_0 < \gamma_1 < 1$ and $n_2$ such that the following holds. For any $n \geq n_2, t \leq \gamma_0 n$ and $p \geq \gamma_1 n$ we have that
\begin{equation}
\label{eq:decrease}
\alpha N(n, p, r, s) > N(n, t, r, s).
\end{equation}
\end{lemma}
\begin{proof}
By definition of $N(\cdot)$ it is enough to prove the following
\begin{equation*}
(1-\alpha)\binom{n-s+1}{r-s+1} < \binom{n-s+1-t}{r-s+1} - \binom{n-s+1-p}{r-s+1}.
\end{equation*}
The left hand side is at most
\begin{equation*}
(1-\alpha)\frac{n^{r-s+1}}{(r-s+1)!},
\end{equation*}
while the right hand side is at least
\begin{equation*}
\frac{1}{(r-s+1)!}\left(\left(n-s+1-t-(r-s)\right)^{r-s+1} - \left((1-\gamma_1)n\right)^{r-s+1}\right).
\end{equation*}
Provided $n_2 \geq \frac{r-1}{\gamma_0}$, this is at least
\begin{equation*}
\frac{n^{r-s+1}}{(r-s+1)!}\left((1-2\gamma_0)^{r-s+1} - (1-\gamma_1)^{r-s+1}\right).
\end{equation*}
Hence it is enough if $1 - \alpha < (1-2\gamma_0)^{r-s+1}-(1-\gamma_1)^{r-s+1}$, which clearly has a solution $(\gamma_0, \gamma_1) \in (0,1)^2$ as desired.
\end{proof}
\begin{lemma}
\label{lem:Ncomp}
Let $r > s \geq 1$ and $\alpha \in (0, 1)$. Then there exist positive constants $\sigma < 1$ and $n_3$ such that the following holds. For any $n \geq n_3$ and $t \leq \sigma n$ it holds that 
\begin{equation}
\label{eq:secDecr}
N(n, t, r, s) - \alpha\binom{n-s}{r-s} < N(n-1, t, r, s).
\end{equation}
\end{lemma}
\begin{proof}
It is equivalent to show
\begin{equation*}
\binom{n-s-t}{r-s} > (1-\alpha)\binom{n-s}{r-s},
\end{equation*}
which holds for $n_3 > \frac{r-1}{\sigma}$ and $\sigma$ such that $(1-2\sigma)^{r-s} > 1-\alpha$.
\end{proof}
We shall frequently use the following inequality, whose simple proof forms part of Lemma 12 in \cite{Keevash08}.
\begin{lemma}
\label{lem:ajutineq}
Suppose $G$ is an $r$-graph, $x$ is a vertex of $G$ and $1 \leq s \leq r-1$. Then
\begin{equation}
\label{eq:rankIneq}
\rank M_s^r(G) \geq \rank M_s^r(G - x)+\rank M_{s-1}^{r-1}(G / x).
\end{equation}
\end{lemma}
Note that in Lemma \ref{lem:ajutineq}, if $G/x$ has no edges, then $G - x$ has the same edge set as $G$ and $\rank M_s^r(G - x) = \rank M_s^r(G)$. We will always assume that for an $r$-graph $G$ with no edges, $\rank M_s^r(G) = 0$.

We will now prove the following equivalent version of Theorem~\ref{thm:main}.
\begin{theorem}
\label{thm:mainNew}
For every $r > s \geq 1$, there exist positive constants $c_0 := c_0(r, s)$ and $n_0 := n_0(r, s)$ such that the following holds. For arbitrary integers 
$n$ and $t$ with $n\geq n_0$ and $1 \leq t \leq c_0n$ and $r$-graph 
$F$ on $n$ vertices, we have that
\begin{itemize}
\item[$(i)$] if $|F| < N(n, t, r, s)$, then $\rank M_s^r(\overline{F}) > \binom{n}{s}-t$
\item[$(ii)$] $|F| = N(n, t, r, s)$ and $\rank M_s^r(\overline{F}) \leq \binom{n}{s}-t$ if and only if $F$ is the upper $(r-s)$-shadow of  an $s$-star configuration of size $t$.
\end{itemize}
\end{theorem}

{\bf Remark.} Observe that part $(i)$ implies that if $F$ is the upper $(r-s)$-shadow of  an $s$-star configuration of size $t$, then $\rank M_s^r(K_n^r - F) = \binom{n}{s}-t$.  Indeed, $|F| = N(n, t, r, s)$, hence 
deleting any one set from
$F$ will make  the rank strictly larger than $\binom{n}{s}-t$, but the addition of a single row can not increase the rank by more than one.

\medskip

{\bf Remark.}
Note that Theorem~\ref{thm:mainNew} indeed implies Theorem~\ref{thm:main} with the same constants $c_0$ and $n_0$. 
Let  $n \geq n_0, 1 \leq t \leq c_0n$ and $G$ be any $r$-graph on $[n]$ with $\rank M_s^r(G) \leq \binom{n}{s} - t$. 
We define $F= \overline{G}$. 
By part $(i)$ of  Theorem~\ref{thm:mainNew}, we must then have $|F| \geq N(n, t, r, s)$, so
$\rex (n,t,r,s) \leq \binom{n}{r} - N(n,t,r,s)$. 
Part $(ii)$ shows that $\rex (n,t,r,s) = \binom{n}{r} - N(n,t,r,s)$ and the
unique family $G$ with $e(G)=\rex(n,t,r,s)$ and $\rank M_s^r(G) \leq \binom{n}{s} - t$ is
the complement of the upper $(r-s)$-shadow of an $s$-star configuration of size $t$.

\begin{proof}[Proof of Theorem~\ref{thm:mainNew}]
Recall first that we have already seen the ``if'' part of $(ii)$ in the introduction: The upper $(r-s)$-shadow of an $s$-star configuration 
of size $t$ has size $N(n,t,r,s)$ and 
for its complement $G=G(n,t,r,s)$ (cf. the discussion before Theorem 5)
$\rank M_s^r(G) \leq \tbinom{n}{s} -t$, since the matrix contains $t$ 
$0$-columns.

For $(i)$ and the ``only if'' part of $(ii)$ we apply double induction, 
first on $s$ and then on a parameter measuring how far the hypergraph is from Theorem~\ref{thm:robust} being applicable. For the second induction
we need to prove a somewhat more technical statement requiring a bit of preparation.

Let $\alpha := \alpha(r, s)$ be the constant given by Theorem~\ref{thm:robust}; we may assume that $\alpha < 1$. 
For every $r$-graph $F'=(V,E)$, we define 
$\Lambda(F', s)$ as 
the smallest integer $\ell\geq 0$, such that 
there exists a (possibly empty) sequence $v_1, \ldots , v_\ell \in V$ of
the vertices, such that for every $i=1, \ldots , \ell$
$v_i$ is of maximum degree in $F_{i-1}=F-\{ v_1, \ldots , v_{i-1}\}$ and
the maximum degree of $F_\ell=F-\{ v_1, \ldots , v_{\ell}\}$ is at most 
$\alpha\binom{|V(F_\ell)|-s}{r-s}$. 
This function is finite, because the maximum degree of $K^r_{r-1}$ 
is $0=\alpha \tbinom{r-1-s}{r-s}$. 
By the definition of $\Lambda$, if $\Lambda(F', s) > 0$ then there exists a vertex
$x\in V(F')$ of maximum degree such that 
$$\Lambda(F',s) = \Lambda(F'-x, s) +1.$$ 
 The r\^ole of $\Lambda(F', s)$ will become clear in Claim~\ref{clm:mainLoop} below, where, intuitively, if zero, it will allow us to appeal to 
Theorem~\ref{thm:robust} immediately, and otherwise it measures 
``how far'' we are from being able to do so.

Let $r > s \geq 1$ be integers.
The base case $s = 1$ will be similar to the induction step, so we will treat them parallel, but will always clearly distinguish which case we deal with.
If $s \geq 2$ we assume that the induction hypothesis (i.e. Theorem~\ref{thm:mainNew}) holds for all $r'>s'$ with $1\leq s' < s$, in particular that 
the appropriate constants $n_0(r-1,s-1)$ and $c_0(r-1,s-1)$ 
exist.

For our technical claim we now define two positive 
constants $n':= n'(r, s)$ and $c' := c'(r, s)$ for every $r>s\geq 1$. 
In the definition we will need $\eps=\eps(r, s)$ and $n_1=n_1(r, s)$, the  two other constants (besides $\alpha$) 
as asserted by Theorem~\ref{thm:robust}, and $\sigma=\sigma(r, s, \alpha)$ and $n_3=n_3(r, s, \alpha)$ the constants in Lemma \ref{lem:Ncomp}. 
For $s = 1$, we set
\begin{align*}
  n' &= \max\{r+s+1, n_1, n_3\}, \\
  c' &< \min\{\eps, \sigma\}.
\end{align*}
For $s > 1$, we set
\begin{align*}
  n' &= \max\{r+s+1, n_1, n_3, n_0(r-1, s-1)+1\}, \\
  c' &< \min\{\eps, \sigma, \frac{1}{2}c_0(r-1, s-1)\}.
\end{align*}

Recall that $s$ is now fixed. The following claim is formulated using the constants defined above.
\begin{claim}
\label{clm:mainLoop}
Let $t\geq 0$ be an integer and $F $ be an $r$-graph on $n$ vertices with $n - \Lambda(F,s) \geq \max \left\{ n', \frac{t}{c'}\right\}$. 
Then 
\begin{itemize}
\item[$(i)$] if $|F| <  N(n, t, r, s)$, then $\rank M_s^r(\overline{F}) > \binom{n}{s}-t$.
\item[$(ii)$] 
if $|F| =N(n, t, r, s)$ and
$\rank M_s^r(\overline{F}) \leq \binom{n}{s}-t$ then $F$ is the upper $(r-s)$-shadow of an $s$-star configuration of size $t$.
\end{itemize}
\end{claim}

Let us discuss shortly how this claim implies the induction statement for $s$.
Recall that $\gamma_0 = \gamma_0(r, s, \alpha) < \gamma_1=\gamma_1(r, s, \alpha)$ and $n_2=n_2(r, s, \alpha)$ are the constants given by Lemma \ref{lem:Nspeed}.
We set
\begin{align*}
n_0=n_0(r, s) &= \max\left\{\frac{n'}{1-\gamma_1}, n_2\right\}, \\
c_0=c_0(r, s) &= \min\{c'(1-\gamma_1), \gamma_0\}.
\end{align*}

Let  $n \geq n_0, 1 \leq t \leq c_0n$ and $F\subseteq \tbinom{[n]}{r}$ 
be from Theorem~\ref{thm:mainNew}. Then $n-n' \geq \gamma_1 n$.
We can also assume that 
$\Lambda(F, s) < \gamma_1 n$, as otherwise 
by the definition of $\Lambda$, \eqref{eq:redef}, and by Lemma \ref{lem:Nspeed},
$$|F| > \alpha \sum_{i=1}^{\gamma_1n} \binom{n-s+1-i}{r-s}=
\alpha N(n, \gamma_1n, r, s) > N(n, t, r, s).$$
Therefore $c'(n - \Lambda(F, s)) > c'(1-\gamma_1)n \geq c_0(r, s)n \geq t$, so the assumptions of Claim \ref{clm:mainLoop} are satisfied, hence 
$(i)$ and $(ii)$ both hold for $F$.

\begin{proof}[Proof of Claim \ref{clm:mainLoop}]
If $t=0$, then $(i)$ is vacuously true, while $(ii)$ holds because the
$\emptyset$ is the upper shadow of itself.

For $t\geq 1$ we prove Claim \ref{clm:mainLoop} by induction on $\Lambda=\Lambda(F, s)$. 
Let $F$ be an $r$-graph on $n$ vertices with $n-\Lambda(F,s) \geq \max \left\{ n', \frac{t}{c'}\right\}$ and $|F| \leq N(n,t,r,s)$.  
Set $G :=\overline{F}$. 

We first consider the base case $\Lambda(F, s) = 0$. Then the maximum degree of $F$ is at most $\alpha \binom{n-s}{r-s}$. So every $s$-subset of $V(G)$ is contained in at most $\alpha \binom{n-s}{r-s}$ elements of $F$. Also note that \eqref{eq:redef} implies
\begin{equation*}
|F| \leq N(n,t,r,s) \leq t\binom{n-s}{r-s} \leq c'n\binom{n-s}{r-s} \leq \eps n \binom{n-s}{r-s}.
\end{equation*}

As $n\geq n' \geq n_1$, the conditions of Theorem \ref{thm:robust} are satisfied. Hence the matrix $M_s^r(G)$ has rank $\tbinom{n}{s}$,
which proves both parts $(i)$ and $(ii)$ because $t\geq 1$.

Let us assume now that $\Lambda(F, s) \geq 1$ and the claim holds for hypergraphs $F'$ with $\Lambda(F', s) < \Lambda(F,s)$. 

Let $x \in [n]$ be a vertex of maximum degree in $F$, such that $\Lambda(F-x,s)=\Lambda(F,s) -1$. 
Hence $$|V(F-x)| - \Lambda(F-x,s) = n - \Lambda(F,s) \geq \max \left\{ n' , \frac{t}{c'}\right\},$$ so 
Claim \ref{clm:mainLoop}  can be applied to $F-x$ and any integer $t', 0\leq t' \leq t$ by the induction on $\Lambda$.

When $s>1$ we can also
apply the induction of Theorem~\ref{thm:mainNew}  for $s-1$ with
 $r-1> s-1$ and $F/x$, as  $|V(F/x)| \geq n'  -1\geq n_0(r-1,s-1)$ and 
for $n \geq 2$ by the  definition of $c'$ we have 
\begin{equation*}
t\leq c'n <  c_0(r-1, s-1)|V(F/x)| = c_0(r-1, s-1)(n-1).
\end{equation*}

We shall now distinguish two cases according to how large $\deg_F(x)$ is.

\medskip
\noindent\textbf{Case~$1$:} $\binom{n-s}{r-s} \leq \deg_F(x) \leq N(n, t, r, s)$.
\medskip

We first deal with the case $s=1$. Then $\deg_F(x) = \binom{n-1}{r-1}$, so $G / x$ is just the empty graph, and hence $\rank M^{r-1}_{0} (G/x) =0$. 
To estimate the rank of the matrix of $\rank M^r_1 (G-x)$ we use Claim~\ref{clm:mainLoop} for
$F-x$ and $t-1$. 
If $|F| < N(n,t,r,1)$, then
\begin{equation*}
|F - x| =|F| -\binom{n-1}{r-1} < N(n, t, r, 1) - \binom{n-1}{r-1} = N(n-1, t-1, r, 1),
\end{equation*}
and by part $(i)$ $\rank M_1^r(G - x) > (n-1) - (t-1) = n-t$.
Lemma \ref{lem:ajutineq} then implies $\rank M_1^r(G) = \rank M_1^r(G - x) > n-t$.
For part $(ii)$ let $|F|=N(n,t,r,1)$ and $\rank M_1^r(G) \leq n-t$.
Then $|F - x| = N(n-1, t-1, r, 1)$ and $\rank M_1^r(G - x)  \leq n-t$.
Hence by part $(ii)$ of  Claim~\ref{clm:mainLoop}, $F - x$ is the upper $(r-1)$-shadow of a $1$-star configuration ${\mathcal S}$ of size $t-1$. 
Because all $r$-subsets of $V(G)$ containing $x$ are edges of $F$, 
$F$ is the upper 
$(r-1)$-shadow of a $1$-star configuration of size $t$ (namely ${\mathcal S} \cup
\{\{ x\}\}$). This finishes the proof of Case 1 when $s=1$.

Now assume $s > 1$. Let $p, 1\leq p \leq t$ be the largest integer such that
\begin{equation*}
|F/x|=\deg_F(x) \geq \sum_{i=1}^{p} \binom{n-s+1-i}{r-s} = N(n, p, r, s) = N(n-1, p, r-1, s-1). 
\end{equation*}

Whenever $p < t$, we can apply part $(i)$ of 
Theorem~\ref{thm:mainNew} to $s-1 < r-1$, $F/x$ and $p+1$.  
Since  $|F/x| < N(n-1, p+1, r-1, s-1)$ by the definition of $p$, we conclude that
\begin{equation}
\label{eq:indR2}
\rank M_{s-1}^{r-1}(G / x) \geq \binom{n-1}{s-1} - p.
\end{equation}

Whenever $|F-x| < N(n-1, t-p, r, s)$ we can apply part $(i)$ of  Claim~\ref{clm:mainLoop} to $F - x$ and $t-p$ 
and conclude that 
\begin{equation}
\label{eq:g-x}
\rank M_s^r(G - x) >\binom{n-1}{s} - (t-p).
\end{equation}

If we assume $|F| <N(n,t,r,s)$ we certainly have $p<t$, as well as
\begin{align}
\label{eq:f-x}
|F - x| & = |F|-|F/x| < N(n,t,r,s) - N(n,p,r,s) =\sum_{i=p+1}^t \binom{n-s+1-i}{r-s} \nonumber \\
& \leq \sum_{i=1}^{t-p}\binom{n-s-i}{r-s} = N(n-1, t-p, r, s),
\end{align}
so both \eqref{eq:indR2} and \eqref{eq:g-x} apply.
Hence by Lemma \ref{lem:ajutineq} we obtain $\rank M_s^r(G) >\binom{n}{s} - t$ and the proof of our part $(i)$ is complete.

For part $(ii)$ we assume that $\rank M_s^r(G) \leq \binom{n}{s} - t$. 
Then $p=t$ or $|F-x|= N(n-1,t-p,r,s)$, otherwise both \eqref{eq:indR2} and \eqref{eq:g-x} apply, contradicting our rank assumption.

Assume first that $|F-x|= N(n-1,t-p,r,s)$. 
In this case we must also have equality everywhere in \eqref{eq:f-x}:
$|F/x|= N(n,p,r,s) = N(n-1,p,r-1,s-1)$ and $p=1$. 
We will arrive at a contradiction.

Applying Theorem~\ref{thm:mainNew} to $s-1 < r-1$, $F/x$ and $p$,
we get that $\rank M_{s-1}^{r-1}(G / x) \geq \binom{n-1}{s-1} - p$ and applying Claim~\ref{clm:mainLoop} to $F - x$ and $t-p$ we have
$\rank M_s^r(G - x) \geq\binom{n-1}{s} - (t-p)$. 
This means that, in order to avoid contradiction with $\rank M_s^r(G) \leq \binom{n}{s} - t$, we must have equalities in both rank-inequalities.
 Then by part $(ii)$ of our induction $F/x$ is the upper $(r-s)$-shadow of an $(s-1)$-star configuration of size $p$ and $F-x$ is the upper $(r-s)$-shadow 
of an $s$-star configuration of size $(t-p)$. 
Now we use that $p=1$: there must be a vertex 
$z\in V(G)\setminus \{ x\}$ which occurs in all  members of $F/x$.  Since $F-x$ is an upper shadow, the degree of $z$ in $F-x$ is at least $1$. 
Hence $z$ has degree at least $|F/x|+1 = \binom{n-s}{r-s}+1$ in $F$. 
As $p=1$, the maximum degree in $F$ should only be $N(n,1,r,s)=\binom{n-s}{r-s}$, a contradiction.

Assume now $p=t$. Then $\deg_F(x) = N(n,t,r,s)= |F|$, so
every edge of $F$ contains $x$. Hence $G - x \simeq K_{n-1}^r$, so by Gottlieb's Theorem we have
$$\rank M_s^r(G - x) = \binom{n-1}{s}$$  
as $n' > r+s$. 

We apply Theorem \ref{thm:mainNew} for $s-1$ with $r-1$, $F/x$, and $t$. By \eqref{eq:mono} $|F/x|=N(n-1,t,r-1,s-1)$, so we conclude that
\begin{equation}
\label{eq:indR}
\rank M_{s-1}^{r-1}(G / x) \geq \binom{n-1}{s-1} - t,
\end{equation}
with equality only if $|F/x| = N(n-1,t,r-1,s-1)=N(n, t, r, s)$ and $F/x$ is the upper $(r-s)$-shadow of an $(s-1)$-star configuration ${\mathcal S}$ of size $t$. By Lemma \ref{lem:ajutineq}, $$\rank M_s^r(G) \geq \rank M_s^r(G-x)+ \rank M_{s-1}^{r-1}(G / x) \geq \binom{n}{s} - t,$$ and for equality to hold we must have equality in \eqref{eq:indR}, that is 
$|F| = |F/x|= N(n, t, r, s)$ and $F$ is the upper $(r-s)$-shadow of an $s$-star configuration of size $t$ ($x$ appended to each member of ${\mathcal S}$). This finishes the proof in Case 1.

\medskip
\noindent\textbf{Case~$2$:} $\alpha\binom{n-s}{r-s} \leq \deg_F(x) < \binom{n-s}{r-s}$.
\medskip

In this case we show that if $|F| \leq N(n,t,r,s)$, then $\rank M_s^r(G) > \binom{n}{s} - t$, concluding the proof of Claim \ref{clm:mainLoop}. 

We first show that $\rank M_{s-1}^{r-1}(G/x) = \binom{n-1}{s-1}$.

If $s=1$, we have $\rank M_{0}^{r-1}(G / x) = 1$, since  
$\deg_F(x) < \binom{n-1}{r-1}$.

If $s>1$  we apply part $(i)$ of
Theorem~\ref{thm:mainNew} to $r-1, s-1$, $F/x$ and $t=1$.
Since $|F/x| = \deg_F(x) < N(n-1, 1, r-1, s-1)$ we 
conclude that $\rank M_{s-1}^{r-1}(G / x) = \binom{n-1}{s-1}$.

For $\rank M_s^r(G - x)$ we apply Claim \ref{clm:mainLoop} to $F - x$ and $t$.
Since $n' \geq n_3$ and $t \leq c'n < \sigma n$, we may apply Lemma \ref{lem:Ncomp} to obtain 
$|F - x| \leq N(n, t, r, s)-\alpha\binom{n-s}{r-s} < N(n-1, t, r, s)$. 
Hence we conclude that $\rank M_s^r(G - x) > \binom{n-1}{s} - t$.

Consequently by Lemma \ref{lem:ajutineq} we obtain $\rank M_s^r(G) > \binom{n}{s} - t$.
\end{proof}

By the remarks following the statement of Claim \ref{clm:mainLoop}, this implies Theorem~\ref{thm:mainNew}.
\end{proof}

\section{\large{Tightness of the main result}}

Fix $r$ and $s$ and suppose $n \geq n_0'(r, s)$ is large enough. 
Recall the definition of $t_{\textrm{max}} := t_{\textrm{max}}(n, r, s)$. It is the maximum value of $t \leq \binom{n}{s}$ such that for any $t' \leq t$, $\binom{n}{r}-\rex(n, t', r, s)$ equals $K(n, t', s, r-s)$, the lower bound given by the Kruskal-Katona theorem for the upper $(r-s)$-shadow of a collection of $s$-subsets of $[n]$ of size $t$.

We will show that $t_{\textrm{max}} < n - r - 1$. The first step is provided by the following lemma.

\begin{lemma}
\label{lem:optim}
Let $r > s \geq 1$. For any $n \geq r+2$ there exists an $r$-graph $R:=R(n, r, s)$ on $n$ vertices with $\rank M_s^r(R) \leq \binom{n}{s} - (n-r-1)$, but $\binom{n}{r}-|R| < N(n, n-(r+1), r, s)$. 
\end{lemma}
\begin{proof}
The proof of this lemma is partially based on a construction of Keevash from \cite{Keevash08}.

We prove the lemma by induction on $s$.

First assume $s=1$.

If $r=2$, we define $R(4, 2, 1) := C_4$. Then $R(4, 2, 1)$ verifies the conditions of the lemma.

If $r>2$, we define $R := R(r+2, r, 1)$ as the complete $r$-graph on the set $[r+2]$, from which all the edges spanned by the vertices $\{1, \ldots, r+1\}$ were removed. Then $\alpha_i := 1, 1 \leq i \leq r+1$, and $\alpha_{r+2} := -(r-1)$, is a dependence sequence for $R$, showing that $\rank M_1^r(R) < r+2$. As $\binom{r+2}{r}-|R| = \binom{r+1}{r} < \binom{r+1}{r-1} = N(r+2, 1, r, 1)$ by \eqref{eq:Ndef}, the graph $R$ also verifies the conditions of the lemma.

For any $n>r+2$, we let $R(n, r, 1)$ have the same edges as $R(r+2, r, 1)$, but extend the vertex set to $[n]$. Then $R(n, r, 1)$ has rank of the associated inclusion matrix at most $r+1$, and
\begin{equation*}
\binom{n}{r} - |R(n, r, 1)| < \binom{n}{r} - \binom{r+1}{r} = N(n, n - (r+1), r, 1).
\end{equation*}
Hence the claim holds for $s=1$.

Now assume $s>1$ and the claim holds for any pair $(s', r')$ with $s'+r' < s+r$.

We define the edge-set of $R(n, r, s)$ as the collection $\binom{[n-1]}{r}$, to which we add all sets $\{n\} \cup e$, where $e$ is an edge in $R(n-1, r-1, s-1)$. 

Note that
\begin{equation}
\label{eq:rkbound}
\rank M_s^r(R(n, r, s)) \leq \binom{n-1}{s} + \rank M_{s-1}^{r-1}(R(n-1, r-1, s-1)).
\end{equation}
Indeed, from the columns corresponding to the $s$-sets containing $n$ we can pick at most $\rank M_{s-1}^{r-1}(R(n-1, r-1, s-1))$ linearly independent ones, to which we may further add at most $\binom{n-1}{s}$ columns, to form a linearly independent collection of size $\rank M_s^r(R(n, r, s))$.

Therefore, 
\begin{equation*}
\binom{n}{s} - \rank M_s^r(R(n, r, s)) \geq \binom{n-1}{s-1} - \rank M_{s-1}^{r-1}(R(n-1, r-1, s-1)) \geq n - (r+1).
\end{equation*}

Furthermore, by construction,
\begin{equation*}
\binom{n}{r} - |R(n, r, s)| = \binom{n-1}{r-1} - |R(n-1, r-1, s-1)| < N(n-1, n - (r+1), r-1, s-1),
\end{equation*}
where the last inequality follows by the induction hypothesis. The proof now follows from \eqref{eq:mono}.
\end{proof}

We are now in position to prove Theorem \ref{thm:sharp}.

\begin{proof}[Proof of Theorem \ref{thm:sharp}]
Let $r > s \geq 1$. We set $n_0'(r, s) := r+2$ and let $n \geq n_0'(r, s)$.

We first claim that for $t=n-r-1$, the lower bound provided by the Kruskal-Katona theorem is still equal to $N(n, t, r, s)$. As $n-r-1 \leq n-s$, by \eqref{eq:tDecomp} we have $t = \binom{n-s}{n-s}+\ldots+\binom{n-s-t+1}{n-s-t+1}$. In the same manner as in the Introduction we obtain that the lower bound provided by the Kruskal-Katona theorem for the upper $(r-s)$-shadow of a collection of $s$-subsets of $[n]$ of size $t$ equals $\binom{n-s+1}{r-s+1}-\binom{n-s+1-t}{r-s+1} = N(n, t, r, s)$, hence the claim.

However, by Lemma \ref{lem:optim}, for $n \geq r+2$ and $t=n-r-1$, $\binom{n}{r}-\rex(n,t,r,s)$ is no longer equal to $N(n, t, r, s)$. Therefore $t_{\textrm{max}} < n-r-1$, as desired.
\end{proof}

\section{\large{Concluding remarks}}
\label{sec:remarks}

We close our discussion with several remarks and questions.

\subsection{\small{Rank-extremal function for graphs and other ranges of $t$}}
In the particular case of $s=1$ and $r=2$ the rank-extremal function as well 
as the rank-extremal graphs can be completely determined for all values of $t \leq n$. 
Indeed, since the rank of a graph $G$ is the sum of the ranks of its components, 
and a component $C$ of $G$ has full rank $|V(C)|$ if it is not bipartite  and has rank $|V(C)|-1$ otherwise, 
one obtains the following. For any graph $G$ on $n$ vertices, $\rank M_1^2(G) = n - b(G)$, where $b(G)$ is 
the number of bipartite components of $G$. With this characterization at hand, it is not 
difficult to compute $\rex(n,t,2,1)$ and the rank-extremal graphs: if $G$ is a rank-extremal graph with $n$ 
vertices and $\rank M^2_1(G)=n-t$, then $G$ consists of $t$ bipartite components and of 
possibly one non-bipartite complete subgraph. All one has to do now is to choose the sizes of components 
such that the number of edges is maximized. We omit the straightforward details, but describe briefly 
$\rex(n,t,2,1)$ for $n\ge 6$. Let $G$ be maximal such that $\rank M_1^2(G) \leq n - t$. Remove all isolated vertices from $G$ to form a new graph $H$. 
The structure of $H$ is now easy to describe. If $t < n-4$ then $H = K_{n-t}$. If $t = n-4$ 
then $H$ is either $K_{2, 3}$ or $K_4$, while if $t = n-3$, $H$ must be $C_4$. Finally, for $t = n-2$, $H$ is 
either $P_2$ or two disjoint edges, for $t = n-1$, $H$ is just $K_2$ and for $t=n$, $H$ is the empty graph. In particular, Theorem~\ref{thm:main} does not hold for $t = n-r-2$, as there are two extremal graphs. 

The case of $n=4, s=1, r=2$ and $C_4$ also shows that Theorem~\ref{thm:main} may fail to hold for small values of $n$, 
even under the additional assumption that $n \geq r+s$. Hence the constant $n_0$ in the theorem statement is necessary. 
It would be interesting to determine the best constant $c_0$ in Theorem~\ref{thm:main}. Our argument only shows it is at most $1$.

Using Theorem~\ref{thm:keevash}, we can determine $\rex(n, t, r, s)$ asymptotically.  
Indeed, for any large enough real $x\in[0,n]$, we have that $\binom{x}{r} \geq \rex(n, \binom{n}{s} - \binom{x}{s}, r, s) > \binom{x-1}{r}$. The upper bound follows from the definition and Theorem~\ref{thm:keevash}, while the lower bound follows by considering the $r$-graph $G$ on $[n]$ and edge set $\binom{[\left \lfloor x \right \rfloor]}{r}$. Consequently the rank-extremal function is asymptotically like $x^r / r!$. However, the error of this approximation becomes very large when $x$ approaches $n$, exactly the case covered by our Theorem~\ref{thm:main}. In general, the exact values of the rank-extremal function as well as the structure of the rank-extremal hypergraphs are still an open problem.

\subsection{\small{Constructive proof of Theorem~\ref{thm:main} for $s=1$}}
The proof of Theorem~\ref{thm:main} is non-constructive in the sense that we only show that 
for an $r$-graph $G$ with more than $\rex(n,t,r,s)$ edges there exist some 
$\tbinom{n}{s}-t+1$ linearly independent rows in the inclusion matrix $M^r_s(G)$. 
In the case $t=s=1$ we can show that $M^r_s(G)$ has full (column) rank by finding 
a particular subgraph of $G$ with $n$ edges, whose rows are linearly independent, as follows. 

Let $H'$ be an $(r-1)$-tight Hamilton cycle on $[n']$ with $n'\equiv r-1\mod r$ and $n'\ge n-r+1$ (this tight Hamilton 
cycle has its edges being any $r$ consecutive elements from $[n']$ modulo $n'$).
The $r$-graph $H$ on the vertex set $[n]$ is obtained by adding to $H'$ $(n-n')$ new vertices $n'+1$, 
\ldots, $n$ and adding new disjoint edges $e_{n'+1}$, \ldots, $e_n$ such that $e_i$ contains vertex $i$ and 
$(r-1)$ further consecutive vertices from the cycle $H'$.
 It is not difficult to see that $M^r_s(H)$ has full rank $n$ (the inclusion matrix of the 
$(r-1)$-tight Hamilton cycle $H'$ on $n'$ vertices has full rank and adding each of the 
edges $e_i$ increases the rank each time by $1$). It follows then from the work 
in~\cite{GlPeWe12} on the extremal number of Hamilton cycles that \emph{any} $r$-graph $G$ on $n$ vertices with $\rex(n,1,r,1)+1=\tbinom{n}{r}-\tbinom{n-1}{r-1}+1$
edges must contain a copy of $H$.

This can be seen as the generalization of the graph argument described above: a graph with $\tbinom{n-1}{2}+1$
edges and $n$ vertices ($n\ge 5$) is clearly connected and not bipartite, yielding that its inclusion matrix has full column rank. It would be interesting to investigate what structures will replace the bipartite components for $r$-graphs when $t>1$ and to look for minimal structures whose inclusion matrices have full rank in the case $r>s>1$.

\begin{comment}
Indeed, \emph{any} $r$-graph $G$ on $n$ vertices and with $\rex{n,1,r,1}+1=\tbinom{n}{r}-\tbinom{n-1}{r-1}+1$
edges contains an $(r-1)$-tight Hamiltonian cycle on $n'$ vertices ($n'\equiv r-1\mod r$ and $n'\ge n-r+1$)
with $n-n'$ further edges attached to it in an arbitrary way 

In~\cite{GlPeWe12}, Glebov, the second author and Weps studied the extremal function 
\[
\ex(n,H):=\max\{|G|\colon G \text{ is an $r$-graph with $n$ vertices and } G\not\supseteq H\}
\] 
for spanning $r$-graphs $H$ being Hamiltonian cycles of some fixed tightness $\l\in[r-1]$, 
where an $\l$-tight Hamiltonian cycle $\Crln$ on $n$ vertices with $(r-\l)| n$ consists of the edges of the form 
$\{i(r-\l)+1,\ldots,i(r-\l)+r\}$ where $i\in\{0,\ldots, n/(r-\l)-1\}$ and addition is modulo $n$ plus $1$. 
It was shown that $\ex(n,\Crln)=\binom{n-1}{r}+\ex(n-1,P(r,\l))$, where $P(r,\l)$ is a link of any 
vertex of $\Crln$.
\end{comment}

\subsection{\small{Random inclusion matrices}}
Lots of research in random graphs dealt recently with questions of transferring 
results from extremal combinatorics to a probabilistic setting. Given a random $r$-graph $G(n,p)$ which is the product probability space 
$\{0,1\}^{\tbinom{n}{r}}$ where $\prob(G)=p^{|G|}(1-p)^{\tbinom{n}{r}-|G|}$. The basic question one might ask 
is the following: what is the threshold for the property that $\rank M^r_s(G(n,p))=\tbinom{n}{s}$? 

In the graph case ($r=2$, $s=1$), the inclusion matrix $M^2_1(G(n,p))$ gets full column rank, exactly at the 
very moment when $G(n,p)$ becomes connected (since it contains a connected 
spanning subgraph with an odd cycle and with $n$ edges  - and its
 inclusion matrix already has full rank).  Further, the threshold $p$ is sharp and equals $\tfrac{\ln n}{n}$. 
For $r>s\ge 1$, $\rank M_s^r(G(n,p))$
stays clearly below $\tbinom{n}{s}$ if there is a zero column. 
It is further easy to see that the threshold $p$ for the property that $M^r_s(G(n,p))$ has no zero columns is precisely at $\tfrac{s (r-s)!\ln n}{n^{r-s}}$, 
which is a standard first moment-second moment calculation. This gives 
us a lower bound on the threshold for $M^r_s(G(n,p))$ to possess full rank. 

On the other hand, it is not hard to prove the upper bound on $p=O(\tfrac{\ln^2 n}{n^{r-s}})$ which is only slightly apart from 
the lower bound. Here is a rough sketch. First, 
from Theorem~\ref{thm:keevash} and some straightforward estimates for binomial coefficients one can obtain the following lemma.
\begin{lemma}\label{lem:many_edges}
 If $G$ is an $r$-uniform hypergraph with $n$ vertices and at least
$\tbinom{n}{r}-tn^{r-s}$ edges, then 
\[
\rank M_s^r(G)\ge \tbinom{n}{s}-c\cdot t,
\]
 where $c>0$ depends on $r$ only. 
\end{lemma}
Next, suppose that we are given $G$ with $n$ vertices and of the rank  $\tbinom{n}{s}- ct-1$. Then, by Lemma~\ref{lem:many_edges},
we obtain $\prob\left[\rank M^r_s(G\cup G(n,p))<\tbinom{n}{s}-c\cdot t\right]\le (1-p)^{tn^{r-s}}=\exp(-ptn^{r-s})$. Thus, starting with 
the empty hypergraph on $n$ vertices, and performing at most $\tbinom{n}{s}$ times multi-round  exposure, we obtain
\begin{equation}\label{eq:finalsum}
 \prob\left[\rank M^r_s(\cup_{i} G(n,p_i))<\tbinom{n}{s}\right]\le \sum_{i=1}^{\tbinom{n}{s}} \exp(-p_i i n^{r-s}).
\end{equation}
Therefore, it is enough to choose $p_i=\tfrac{C\log n}{i n^{r-s}}$ for $C$ sufficiently large, so that the sum in~\eqref{eq:finalsum} is $o(1)$.
 Moreover, since $\sum_{i=1}^{\tbinom{n}{s}}p_i=O(\tfrac{\log^2n}{n^{r-s}})$, 
the upper bound on the threshold for $M^r_s(G(n,p))$ to possess full rank follows.

Finally, it should be mentioned, that  Friedgut's theorem~\cite{Fri99} implies that the property $\rank M^r_s(G(n,p))=\tbinom{n}{s}$ has a sharp threshold, 
since the phase transition happens between $\tfrac{\ln n}{n^{r-s}}$ and $\tfrac{\ln^2 n}{n^{r-s}}$ and thus $p$ does not have a rational exponent. It would be interesting to determine the precise threshold for the property above. 
 In fact, we believe it should match the lower bound.
\subsection{\small{Local resilience of the rank}}

The study of graph resilience was initiated by Sudakov and Vu in \cite{SudVu08}. In view of Theorem \ref{thm:robust}, one can formulate a similar question about the rank of $K_n^r$. 

Let $r \geq s \geq 0$. If $F \subseteq \binom{[n]}{r}$ and $S \in \binom{[n]}{s}$ we define $\deg_F(S)$ as the number of elements of $F$ containing $S$. Then one has the following problem.

\begin{problem}
For $n$ sufficiently large, determine the maximum value $m(n, r, s)$ such that for any $F \subseteq \binom{[n]}{r}$ with $\deg_F(S) \leq m(n, r, s), \forall S \in \binom{[n]}{s}$, we have $\rank M_s^r(K_n^r - F) = \binom{n}{s}$.
\end{problem} 

Theorem \ref{thm:robust} shows that for $r \geq 2s - 1$ we have $m(n, r, s) \geq cn^{r+1-2s}$, for some $c > 0$.

\subsection{\small{Fields of finite characteristic}}
It turns out that the inclusion matrix of the complete hypergraph does not necessarily possess the full rank. In the case of $\GF(2)$ the rank was computed 
 by Linial and Rothschild~\cite{LinRot81} and in the case of $\GF(p)$ by Wilson~\cite{Wilson_m}, see also Frankl~\cite{Fra90}.

\section*{Acknowledgements} We would like to thank Nati Linial for his interest in this project and for  
bringing the results in finite fields to our attention.

\providecommand{\bysame}{\leavevmode\hbox to3em{\hrulefill}\thinspace}
\providecommand{\MR}{\relax\ifhmode\unskip\space\fi MR }
\providecommand{\MRhref}[2]{%
  \href{http://www.ams.org/mathscinet-getitem?mr=#1}{#2}
}
\providecommand{\href}[2]{#2}


\begin{thebibliography}{10}

\bibitem{Ahls03}
Rudolf Ahlswede, Harout Aydinian, and Levon Khachatrian, \emph{Maximum number
  of constant weight vertices of the unit $n$-cube contained in a
  $k$-dimensional subspace}, Combinatorica \textbf{23} (2003), no.~1, 5--22.

\bibitem{Chow2010}
Ameera Chowdhury and Bal\'asz Patk\'os, \emph{Shadows and intersections in
  vector spaces}, J. Combin. Theory Ser. A \textbf{117} (2010), no.~8,
  1095--1106.

\bibitem{Fra90}
Peter Frankl, \emph{Intersection theorems and mod {$p$} rank of inclusion
  matrices}, J. Combin. Theory Ser. A \textbf{54} (1990), no.~1, 85--94.

\bibitem{FranklToku91}
P\'eter Frankl and Norihide Tokushige, \emph{The {K}ruskal–{K}atona theorem,
  some of its analogues and applications}, Extremal Problems for finite sets
  (P\'eter Frankl, Zolt\'an F{\"u}redi, Gyula O.~H. Katona, and Desz\H{o}
  Mikl\'os, eds.), Bolyai Soc. Math. Stud., vol.~3, J\'anos Bolyai Mathematical
  Society, 1991, pp.~229--250.

\bibitem{Fri99}
Ehud Friedgut, \emph{Sharp thresholds of graph properties, and the k-sat
  problem}, J. Amer. Math. Soc. \textbf{12} (1999), no.~4, 1017--1054.

\bibitem{GlPeWe12}
Roman Glebov, Yury Person, and Wilma Weps, \emph{On extremal hypergraphs for
  {H}amiltonian cycles}, European J. Combin. \textbf{33} (2012), no.~4,
  544--555.

\bibitem{Gottlieb}
Daniel~Henry Gottlieb, \emph{A certain class of incidence matrices}, Proc.
  Amer. Math. Soc. \textbf{17} (1966), no.~6, 1233--1237.

\bibitem{Kat68}
Gyula Katona, \emph{A theorem of finite sets}, Theory of graphs (Proc. Colloq.,
  Tihany, 1966), Academic Press, New York, 1968, pp.~187--207.

\bibitem{Keevash08}
Peter Keevash, \emph{Shadows and intersections: stability and new proofs}, Adv.
  Math. \textbf{218} (2008), no.~5, 1685--1703.

\bibitem{Addendum}
\bysame, \emph{Addendum to \textit{{S}hadows and intersections: stability and
  new proofs}}, January 2010, Available at
  \url{http://www.maths.qmul.ac.uk/~keevash/papers/kk-addendum.pdf}.

\bibitem{Keevash11}
\bysame, \emph{{H}ypergraph {T}ur{\'a}n problems}, Surveys in {C}ombinatorics,
  London Math. Soc. Lecture Notes Series, Cambridge University Press, 2011,
  pp.~83--140.

\bibitem{Kru63}
Joseph~B. Kruskal, \emph{The number of simplices in a complex}, Mathematical
  Optimization Techniques, Univ. of California Press, Berkeley, Calif., 1963,
  pp.~251--278.

\bibitem{LinRot81}
Nathan Linial and Bruce~L. Rothschild, \emph{Incidence matrices of subsets---a
  rank formula}, SIAM J. Algebraic Discrete Methods \textbf{2} (1981), no.~3,
  333--340.

\bibitem{Longstaff77}
William~E. Longstaff, \emph{Combinatorial solution of certain systems of linear
  equations involving $(0, 1)$ matrices}, J. Austral. Math. Soc. \textbf{23}
  (1977), no.~3, 266--274.

\bibitem{Lovasz79}
L{\'a}szl{\'o} Lov{\'a}sz, \emph{Combinatorial problems and exercises},
  North-Holland Publishing Co., Amsterdam, 1979.

\bibitem{Odlyzko81}
Andrew~Michael Odlyzko, \emph{On the ranks of some $(0, 1)$-matrices with
  constant row sums}, J. Austral. Math. Soc. \textbf{31} (1981), no.~2,
  193--201.

\bibitem{SudVu08}
Benny Sudakov and Van Vu, \emph{Local resilience of graphs}, Random Struct.
  Alg. \textbf{33} (2008), no.~4, 409--433.

\bibitem{Wilson_m}
Richard~M. Wilson, \emph{The {S}mith normal form of inclusion matrices}, 1980,
  manuscript.

\end{thebibliography}
\end{document}